\documentclass{amsart}
\usepackage{amssymb}
\usepackage{amsfonts}

 \newtheorem{thm}{Theorem}[section]
 
 \newtheorem{lem}[thm]{Lemma}
 
\theoremstyle{definition}
 \newtheorem{defn}[thm]{Definition}
\theoremstyle{remark}
\newtheorem*{ex}{Example}
\numberwithin{equation}{section}

\begin{document}

%
%
%
%
%
%
%
%
%

\title[Structural theorems for ultradistribution semigroups]
 {Structural theorems for ultradistribution semigroups}

\author[Marko Kosti\' c]{Marko Kosti\' c}

\author{Stevan Pilipovi\' c}

\author{Daniel Velinov}


%


\begin{abstract}
We consider exponential  ultradistribution semigroups with
non--densely defined generators and  give structural theorems for
ultradistribution semigroups. Also structural theorems for
exponential ultradistribution semigroups are given.
\end{abstract}

\maketitle
\section{Introduction and preliminaries}\label{intro}

\noindent In the previous paper
\cite{kps} the first two authors have analyzed, on the basis of well developed theory,
ultradistribution semigroups through the existence of an  ultradistributional fundamental solution
to $u_t-Au=\delta$, where $A$ is the corresponding closed operator as well as through the sub-exponential estimate of the resolvent $||R(\lambda, A)||\leq Ce^{M(k|\lambda|)}$ in an appropriate domain
defined by the associated function $M$.  In Theorems 8 and 9 of \cite{kps} they have given examples based on these characterizations. In this  paper
we give complete structural characterization of
 ultradistribution semigroups with the aim of their full characterizations and connections with the corresponding Cauchy problems.

 As we mentioned, the  literature
related to   ultradistribution
semigroups is pretty reach. It is based on the
generalizations of $C_0-$semigroups,
especially  of various classes of integrated
semigroups of W. Arendt \cite{a11} and further
extensions,   and \cite{me152}
(see also  \cite{a22}, \cite{l115},
\cite{keya}, \cite{l114},
\cite{n181}) Especially, we refer to the excellent  monograph \cite{a43} and the references
therein.
Ultradistribution semigroups with densely defined
generators were  considered by J. Chazarain in \cite{cha}
(see also \cite{b42}, \cite{ci1}, \cite{keya}
and references therein) while
H. Komatsu \cite{k92}  considered ultradistribution semigroups
with non-densely defined generators as well as Laplace
hyperfunction semigroups. We also refer to R. Beals
\cite{b41}-\cite{b42} for the theory of $\omega$-ultradistribution
semigroups with the densely defined generators, to P. C. Kunstmann
\cite{ku113} and  to the monograph of I. Melnikova and A. Filinkov
\cite{me152} for ultradistribution semigroups with the non-densely
defined generators and applications to abstract Cauchy problem
(see \cite{me151}, \cite{me149}).
In \cite{kps} are analyzed ultradistribution semigroups
following the approaches of
P. C. Kunstmann \cite{ku112} and S. Wang \cite{w241},
where distributions semigroups are considered. The most recent theory
of ultradistribution semigroups is given in the monograph of M. Kostic \cite{kosticknjiga}. \\
\indent

We recall in Section 2 some of  definitions and results from
\cite{kps} related to ultradistribution semigropups.
Ultradifferentiable operators are used in order
to clarify relations between
exponentially bounded and tempered ultradistribution
semigroups and convoluted semigroups. \\
\indent In Section 3 we give a structural characterizations
for ultradistribution semigroups. The main results are
given in Theorem \ref{struc}. We give five conditions for
ultradistribution semigroups and the corresponding five
conditions for exponential ultradistribution semigroups
and we give relations between them.

\subsection{Notation from ultradistribution theory}

Here we use the same notation like in \cite{kps} and we follow
approach of H. Komatsu \cite{k91} in defining ultradistribution spaces.
%
%
%
%
If
$(M_p)$ verifies (M.1), (M.2) and (M.3)', then the
spaces of Beurling, respectively,
Roumieu ultradifferentiable functions, are
${\mathcal D}^{(M_p)}({\mathbb R})$ and
${\mathcal D}^{\{M_p\}}({\mathbb R})$
With the notation $*$ for both cases of brackets, we define
${\mathcal D}'^{*}({\mathbb R},E):= L({\mathcal D}^{*}({\mathbb R}),
E)$ as the space of continuous linear functions from ${\mathcal
D}^{*}({\mathbb R})$ into $E$; ${\mathcal D}^{*}_0({\mathbb R})$
denotes the space of elements in ${\mathcal D}^{*}({\mathbb R})$
which are supported by $[0,\infty)$ while ${\mathcal E}'^{*}_0$
denotes the space of ultradistributions whose supports are compact
subsets of $[0,\infty)$. We also use the traditional notation
${\mathcal D}'^{*}_+({\mathbb R},E)$ for the space of vector valued
ultradistributions supported by $[0,\infty).$ We refer to
\cite{k82} for the basic material related
to vector-valued ultradistribution spaces.

%
%
Spaces of tempered ultradistributions of Beurling and Roumieu
type are given in \cite{pilip} (see also \cite{sp94}) as duals of the test
spaces ${\mathcal S}^{(M_p)}({\mathbb R})$ and
${\mathcal S}^{\{M_p\}}({\mathbb R})$, respectively.

Recall (\cite{k91}), an entire function of the form
$P(\lambda)=\sum_{p=0}^{\infty}a_p\lambda^p,\;\ \lambda \in {\mathbb
C},$ is of $(M_p)$-class, respectively, of $\{M_p\}$-class, ( i.e.,
an ultrapolynomial of the respective class) if there exist $k>0$ and
$C>0$, resp., for every $k>0$ there exists a constant $C>0,$ such
that $|a_p|\leq Ck^p/M_{p},\ p\in{\mathbb N}.$ The corresponding
ultradifferential operator $P(d/dt)= \sum_{p=0}^{\infty}a_p
d^p/dt^p$ is of $(M_p)$-class, respectively
of $\{M_p\}$-class.
The composition and the sum
of ultradifferential operators of the Beurling, resp., the Roumieu
class, are ultradifferential operators of the Beurling, resp., the
Roumieu class.

The following assertion is well known in the theory of
ultradistributions (cf. \cite{k91} and ~\cite[Theorem
4.7]{k92}).

{\it Let $T\in{\mathcal D}_+^{'*}(\mathbb{R},E)$. Then for every
$a>0$ there exist an ultradifferential operator of $(M_p)$-class,
formally of the form
\begin{equation} \label{str1}
P_L(d/dt)=\prod_{p=1}^\infty\Bigl(1+\frac{L^2}{m_p^{2}}d^2/dt^2\Bigr)
=\sum_{p=0}^\infty a_pd^p/dt^p,
\end{equation}
where $L>0$ is some constant, resp., of $\{M_p\}$-class, formally of
the form
\begin{equation} \label{str2}
P_{L_p}(d/dt)=\prod_{p=1}^\infty\Bigl(1+\frac{L_p^2}{m_p^{2}}d^2/dt^2\Bigr)=
\sum_{p=0}^\infty a_pd^p/dt^p,
\end{equation}
where $ (L_p)_p$ is a sequence decreasing to $0,$ and a continuous
function $f:(-a,a)\rightarrow E$ such that
$$ T=P_L(-id/dt)f,\;
\mbox{ on }
{\mathcal D}^{(M_p)}((-a,a)),\mbox{ in } (M_p)-\mbox{case, resp.},$$
$$
T=P_{L_p}(-id/dt)f,\; \mbox{ on } {\mathcal D}^{\{M_p\}}((-a,a)),
\mbox{ in } \{M_p\}-\mbox{case}.
$$
Due to \cite[Theorem 2]{pilip}, we have the
following representation theorems for tempered ultradistributions in
the case when (M.1), (M.2) and (M.3) are valid.

{\it Let $T\in{\mathcal S}_+^{'*}(\mathbb{R},E)$. Then  there exist
an ultradifferential operator of $(M_p)$-class, $P_L(d/dt),\ L>0,$
formally of the form $(\ref{str1}),$ resp., of $\{M_p\}$-class,
$P_{L_p}(d/dt),$ $ (L_p)_p$ is a sequence tending to zero, formally
of the form $(\ref{str2}),$ and a continuous function
$f:\mathbb{R}\rightarrow E$ with the properties supp$f\subset
(-a,\infty),$ for some $a>0,$ $||f(t)||\leq Ae^{M(k|t|)},\ t\in
\mathbb{R}$, for some $k>0$ and $A>0$, resp., for every $k>0$ and a
corresponding $A>0$, and that $ T=P_L(-id/dt)f\;
\mbox{ in } (M_p)$-case on ${\mathcal S}^{(M_p)}(\mathbb{R}),$
resp., $ T=P_{L_p}(-id/dt)f\; \mbox{ in } \{M_p\}$-case on
${\mathcal S}^{\{M_p\}}(\mathbb{R}).$ }


\section{Ultradistribution semigroups}

\subsection{Some results from ultra distribution theory}

We will consider ultradistribution semigroups in the framework
of exponential ultradistributions which we define through tempered
ultradistributions.We assume here that $(M_p)$ satisfies (M.1), (M.2)
and (M.3). The purpose of (M.3) is again the use of ~\cite[Theorem
4.8]{k82}.

\begin{defn}\label{prostor}
Let $a\geq 0.$ Then $
{\mathcal{SE}}^{*}_{a}({\mathbb R}):=\{\phi \in C^\infty({\mathbb
R}) :  e^{a\cdot}\phi\in{\mathcal S}^{*}(\mathbb R)\}.$

The convergence in this space is given by
$$\phi_n\rightarrow 0 \mbox{ in }
{\mathcal{SE}}^{*}_{a}({\mathbb R}) \mbox{ iff }
e^{a \cdot}\phi_n\rightarrow 0 \mbox{ in } {\mathcal
S}^{*}({\mathbb R}).$$

We denote by $\mathcal{SE}'^{*}_{a}(\mathbb{R},E)$ the space
of all
continuous linear mappings from $\mathcal{SE}^{*}_{a}(\mathbb{R})$
into $E$ equipped with the strong topology.
\end{defn}

We have
$$
 F\in {\mathcal{SE}}'^{*}_{a}(\mathbb{R},E) \mbox{
iff } e^{-a\cdot}F\in {\mathcal{S}}'^{*}(\mathbb{R},E).
$$

\begin{thm}\label{novo}
Let $G \in \mathcal{SE}'^{*}_{a}(\mathbb{R},E).$ Then there exists
an ultrapolynomial $P$ of $*$-class and a function $g\in
C({\mathbb R},E)$ with the property that there exist $k>0$ and
$C>0$, resp., for every $k>0$ there exists an appropriate $C_k>0$
such that
$$
e^{-ax}||g(x)||\leq C_ke^{M(k|x|)}, \;x\in \mathbb{R}\; \mbox{ and }\;
 G=P(d/dt)g.
$$
\end{thm}
\begin{proof} Let us prove the assertion in the Beurling case.
Since \\ $e^{-a\cdot}G\in \mathcal{S}'^{(M_p)}(\mathbb{R},E),$ one can
use the same arguments as in \cite{pilip} in order to see that there
exist an ultrapolynomial $P$ of $(M_p)$-class and a function $g_1\in
C({\mathbb R},E)$ with the property that there exist $k>0$ and
$C_k>0$ such that
$$ ||g_1(x)||\leq C_ke^{M(k|x|)}\mbox{ and that }
 G=e^{ax}P(d/dt)g_1(x).
$$
Put $g(x)=e^{ax}g_1(x),\ x\in {\mathbb R}.$ By Leibnitz formula, we
have
$$e^{ax}P(d/dt)g_1(x)=\sum_{j=0}^{\infty}
\Bigl(\sum_{k=0}^{\infty}{j+k \choose j}(-1)^ka^ka_{k+j}
\Bigr)(e^{ax}g_1(x))^{(j)}
$$
and we will prove the assertion if we show that $b_{j}\leq
C\frac{L^{j}}{M_{j}},\ j\in \mathbb{N}_0, $ for some $C, \ L>0,$
where $b_j=\sum_{k=0}^{\infty}{k+j \choose j}a^ka_{k+j},\;
j\in\mathbb{N}_0.$

We will use the following inequality,
${j+k \choose j}\leq 2^{k+1}k^ke^j,\; j,k \in \mathbb{N}.$

This follows from
$${j+k \choose j}\leq (j+k)^k\leq 2^kj^k+2^kk^k\leq
2^k(k^ke^j+k^k)=2^kk^k(e^j+1),\; j,k \in \mathbb{N},
$$
where we use
$j^k\leq k^ke^j,\; j,k \in \mathbb{N}.$ This is clear for $k\geq j$.
Let us prove this for $k<j$. Put $k=\varepsilon j$ and note,
if $\varepsilon\in (0,1)$, then $\varepsilon \ln \varepsilon\in(-1,0)$
and
$$\varepsilon j\ln j\leq \varepsilon j \ln j+\varepsilon j\ln \varepsilon
+j.$$
 This implies $j^k\leq k^ke^j, k<j$. Now we will
 estimate $b_j$ using the estimate
$$|a_{k+j}|\leq C \frac{h^{k+j}}{M_{k+j}}
\mbox{ for some } \; h>0,\; C>0$$
and that for every $\varepsilon>0$ there
exists $C_\varepsilon>0$ such that
$M_jk^k\leq C_\varepsilon \varepsilon^{k+j}M_{k+j}.$\\
\indent
With this we have
$$M_j|b_j|\leq
2 \sum_{k=0}^{\infty}\frac{h^{k+j}M_j2^kk^ke^ja^k}{M_{k+j}}
\leq 2C(he)^j\sum_{k=0}^{\infty}
\frac{(2ha)^k M_jk^k}{M_{k+j}}, j\in \mathbb{N}
$$
and choosing $\varepsilon$ enough small,
we obtain the convergence of the last series.
This implies that there exist $L>0$ and $C>0$ such that
$|b_j|\leq CL^j/M_j,\; j\in \mathbb{N}.$

\end{proof}

We need the following estimations of ultrapolynomials:

\begin{lem} \label{ocena}
\begin{itemize}
\item[(a)] Let $P_L$ be of the form $(\ref{str1})$. Then there exist
$C,\ C_1>0,\ L_1,\ L_2>0$ such that
$$e^{2M(L|\zeta|)}\leq |P_L(\zeta)|\leq Ce^{M(L_1|\zeta|)}\mbox{
if }\, |Im\zeta| < \frac{|Re \zeta|}{2}+ \frac{1}{L}$$ and $|a_p|\leq
C_1 L_2^p/M_p,\; p\in {\mathbb N}_0.$

\item[(b)] Let $(L_p)_p$ be a sequence which strictly decreases to zero
and $P_{L_p}$ be defined by $(\ref{str2})$. Then there exists
$C>0$ such that, for every  $k>0,$ there exists  $C_k>0,$
such that
$$ |P_{L_p}(\zeta)|\leq C_ke^{M(k|\zeta|)}\mbox{
if }\,
|Im\zeta| < \frac{|Re \zeta|}{2}+ \frac{1}{L_1},
$$
and (with  another $C_k,$ for given $k>0$) $|a_p|\leq C_k
k^p/M_p,\; p\in {{\mathbb N}_{0}}. $ Moreover, there exists a
subordinate function $\varepsilon(\rho),\;\rho\geq 0,$ such that
$$e^{2M(\varepsilon(|\zeta|))}\leq |P_{L_p}(\zeta)|\mbox{
if }\,
|Im\zeta| <\frac{|Re \zeta|}{2}+ \frac{1}{L}.$$
\end{itemize}
\end{lem}

\begin{proof}
We will prove only the part
$$e^{2M(L|\zeta|)}\leq |P_L(\zeta)|\mbox{
if }\, |Im\zeta| < \frac{|Re \zeta|}{2}+ \frac{1}{L}.
$$
Note that for any $c>0$, the inequality $x^2-y^2\geq 0$
($\zeta=x+iy$) implies $|1+c\zeta^2| \geq |c\zeta|^2.$ Also,
$|1+c\zeta^2|>|c\zeta|^2,$ for all sufficiently small $|\zeta|.$
Thus, by the simple calculation we have that
$$ \bigl|1+\frac{L^2\zeta^2}{m_p^2}\bigr|\geq
\frac{L^2}{m_p^2}|\zeta|^2\;\mbox{
if }\,
 |Im\zeta| < \frac{|Re \zeta|}{2}+ \frac{1}{L}.
 $$
This implies
$$ |P_L(\zeta)|=\bigl|\prod_{p=1}^\infty
\Bigl(1+\frac{L^2}{m_p^{2}}\zeta^2\bigr)\bigr|\geq
\prod_{p=1}^\infty\Bigl(\frac{L^2}{m_p^{2}}|\zeta|^2\Bigr)\geq
e^{2M(|\zeta|)}\;\mbox{
if }\,  |Im\zeta| < \frac{|Re \zeta|}{2}+ \frac{1}{L}.
$$
\end{proof}

\begin{lem}\label{izo}
Let $P_L(d/dt)$ and $P_{L_p}(d/dt)$ be of the form $(\ref{str1})$
and $(\ref{str2})$, respectively. The mappings
$$P_L(id/dt): {\mathcal S}^{(M_p)}({\mathbb R})
\rightarrow {\mathcal S}^{(M_p)}({\mathbb R}),\;\;
\phi\mapsto P_L(id/dt)\phi,$$
$$P_{L_p}(id/dt): {\mathcal S}^{\{M_p\}}({\mathbb R})
\rightarrow {\mathcal S}^{\{M_p\}}({\mathbb R}),\;\;
\phi\mapsto P_{L_p}(id/dt)\phi,
$$
are continuous linear bijections.
\end{lem}
\begin{proof}
We will prove the lemma in the Beurling case. Let
$\phi\in{\mathcal S}^{(M_p)}({\mathbb R}).$ Then $${\mathcal
F}(P_L(id/dt)\phi)(\xi)
=P_L(-\xi)\hat{\phi}(\xi)=P_L(\xi)\hat{\phi}(\xi),\;
\xi\in{\mathbb R}.$$ One can prove by standard arguments that
$P_L(\xi)\hat{\phi}\in{\mathcal S}^{(M_p)}({\mathbb R}).$ We have
to prove that $\hat{\phi}/P_L(\xi)\in{\mathcal S}^{(M_p)}({\mathbb
R}).$

Notice that there exists $r>0$ such that, for every $\xi\in {\mathbb
R},$ the circle $k_\xi(r)$, with the center $\xi$ and the radius
$r,$ is contained in the domain $|Im \zeta|<1/C$ where the estimates
of Lemma \ref{ocena} are satisfied. By Cauchy's formula, with
suitable constants, it follows
$$|(P_L^{-1})^{(n)}(\xi)|\leq C \frac{n!}{r^n}\sup
\{|P_L^{-1}(\xi+re^{i\theta})| : \theta\in[0,2\pi]\}\leq
$$
$$\leq  C \frac{n!}{r^n}
e^{M(L(|\xi|+r))}\leq  C_1 \frac{n!}{r^n} e^{M((L+1)|\xi|)},\;
\xi\in \mathbb{R},\; n\in {{\mathbb N}_{0}}.
$$
Now it is easy to prove that for every $h>0$,
$$\sup\Bigl\{\frac{h^n|(\hat{\phi}/P_L)^{(n)}(\xi)
|e^{M(h|\xi|)}}{M_n} :  \xi\in {\mathbb R},
\;n\in {{\mathbb N}_{0}}\Bigr\} < \infty$$ which is equivalent with
$\hat{\phi}/P_L\in{\mathcal S}^{(M_p)}({\mathbb R}).$
\end{proof}

\subsection{Structural theorems}

Let $A$ be a closed operator and $K$ be a locally integrable
function on $[0,\tau )$, $0<\tau \leq \infty $, and let
$\Theta(t):=\int_0^t K(s)\, ds$, $0\leq t\leq\tau$. Recall
(see \cite{kosticknjiga}, \cite{mn}, \cite{l114} for example), if there exists a
strongly continuous operator family $(S_K(t))_{t\in [0,\tau)}$ such
that $S_{K}(t)C=CS_{K}(t),$ $S_{K} (t)A \subset
AS_{K}(t),$ $ \int_{0}^{t} S_{K}(s)x\, ds\in D(A)$, for $t\in[0,\tau)$, $\;x\in E$
and
$$
A\int_{0}^{t} S_{K}(s)x\, ds = S_{K}(t)x-\Theta (t)Cx, \;x\in
E,
$$ 
then $(S_K(t))_{t\in [0,\tau)}$ is called a (local) $K$-convoluted
$C$-semigroup having $A$ as a subgenerator.

If $\tau=\infty$, then it is said that
$(S_{K}(t))_{t\geq 0}$ is an exponentially bounded $K$-convoluted $C-$semigroup
generated by $A$ if, additionally, there exist $M>0$ and $\omega \in
{\mathbb R}$ such that $||S_{K}(t)||\leq Me^{\omega t},\ t\geq 0.$
 $(S_{K}(t))_{t\in [0,\tau)}$ is called non-degenerate,
if the assumption $S_{K}(t)x=0,$ for all $t\in [0,\tau),$ implies
$x=0.$

We recall  from \cite{kps} the definitions of
L-ultradistribution semigroups and
ultradistribution semigroups
(following \cite{ku112} and \cite{w241})
and define exponential ultradistribution semigroups.
\begin{defn} \label{d6.1}
Let $G\in {\mathcal D}'^{*}_+({\mathbb R},L(E)).$ It is an
exponential L-ultradistribution semigroup of $*$-class
if the following conditions $(U.1)$--$(U.5)$ hold:
\newline
$ (U.1)\;\;  G(\phi*\psi) = G(\phi)G(\psi),\ \phi,\ \psi \in
{\mathcal D}^{*}_0({\mathbb R});\; $
\newline
$ (U.2)\;\;    {\mathcal N} (G) := \bigcap_{\phi\in {\mathcal
D}^{*}_0({\mathbb R})} N (G(\phi)) = \{0\};\;$
 \newline
 $ (U.3)\;\;   {\mathcal
R}(G) :=
 \bigcup_{\phi \in {\mathcal D}^{*}_0({\mathbb R})}
R(G(\phi)) \mbox{ is dense in } E;$
\newline
$ (U.4)$  For every $x \in {\mathcal R}(G)$ there exists a function
$u \in C([0,\infty),E)$ satisfying
$$
u(0) = x \mbox{ and } G(\phi)x =
\int_{0}^{\infty}\phi(t)u(t)\, dt,\ \phi \in {\mathcal
D}^{*}({\mathbb R}).$$
\newline
$ (U.5)$ There exists $a\geq0$ such that
$G\in{\mathcal{SE}}'^{*}_{a}(\mathbb{R},L(E));$
\newline

Recall, $f\ast_{0}g(t):=
\int_{0}^{t}f(t-s)g(s)\, ds,\; t\in\mathbb{R}.$
If $G\in {\mathcal D}'^{*}_+({\mathbb R},L(E))$ satisfies
\newline
$ (U.6)\;\;   G(\phi*_0\psi) = G (\phi)G(\psi),\,
\mbox{for}\,\, \,  \phi,\ \psi \in
{\mathcal  D}^{*}({\mathbb R})\;$, and $(U.5)$, then it is a exponential
pre-ultradistribution semigroup, in short, pre-(EUDSG) of
$*$-class.

 If $(U.6)$, $(U.5)$ and $(U.2)$
are fulfilled for $G$, then $G$ is
an exponential ultradistribution
semigroup of $*$-class, in short, (EUDSG). A
pre-(EUDSG) $G$ it is said that is dense if
additionally $(U.3)$ is satisfied.

If only $(U.6)$ holds then we call $G$
pre-ultradistribution semigroup or pre-(UDSG).

If $(U.6)$ and $(U.2)$ holds, $G$ is
ultradistribution semigroup, in short (UDSG),
and if additionally $(U.3)$ holds then $G$
is dense ultradistribution semigroup.
\end{defn}
If $G\in {\mathcal D}'^{*}_+({\mathbb R},L(E))$, then the condition:

\noindent
$ (U.2)' \;\;\;\; \mbox{supp}G(\cdot)x \nsubseteq \{0\},\; \;
\mbox{for
every } x\in E \setminus \{0\},$
\noindent is equivalent to $(U.2).$

 Let $D$ be another Banach space and $P \in {\mathcal
D}'^{*}_{+}({\mathbb R},L(D,E))$. Then, as in the case of distribution semigroups,
 $G \in {\mathcal
D}'^{*}_{+}({\mathbb R},L(E,D))$ is an ultradistribution
fundamental solution for $P$
if
$$
P \ast G = \delta \otimes I_{E} \mbox{   and   } G \ast P = \delta
\otimes I_{D}.
$$
If additionally  $G\in {\mathcal{SE}}'^{*}_{a}(\mathbb{R},L(E,[D(A)])),$ holds
for some $a\geq 0$, then it is said that $G$
is exponential ultradistribution fundamental solution for $P$.

As  in the case of distributions, an ultradistribution fundamental
solution for $P \in {\mathcal D}'^{*}_{+}({\mathbb R},L(D,E))$ is
uniquely determined.\\
\indent

In the sequel, we will use the phrase \textquotedblleft $G$ is an
ultradistribution fundamental solution for $A$'' if $G$ is an
ultradistribution fundamental solution for $P:=\delta' \otimes
I_{D(A)}-\delta \otimes A \in {\mathcal
D}'^{*}_{+}(\mathbb{R},L([D(A)],E)).$

Following the investigation of H. Komatsu \cite{k92}, in the
framework of Denjoy-Karleman-Komatsu theory of ultradistributions
and P. C. Kunstmann \cite{ku113}, in the theory of
$\omega$-ultradistributions, we define the next regions:

$\Omega^{(M_p)} :=\{\lambda \in {\mathbb C} :  Re \lambda \geq
M(k|\lambda|)+ C\},\; \mbox{ for some } k>0,\; C>0, $ resp.,

$\Omega^{\{M_p\}} :=\{\lambda \in {\mathbb C} : Re \lambda \geq
M(k|\lambda|)+ C_k\},\; \mbox{ for every } k>0$ and a corresponding
$C_k>0.$ By $\Omega^{\ast}$ is denoted either $\Omega^{(M_p)}$ or
$\Omega^{ \{M_p\} }.$

In Theorem \ref{struc} which is to follow, in the case of tempered
ultradistribution semigroups (and similarly in the case of
exponentially bounded ultradistribution semigroups), we use
~\cite[Theorem 3.5.14]{kosticknjiga}, where the
inverse Laplace transform is performed on the straight line
connecting points $\bar{a}-i\infty$ and $\bar{a}+i\infty,$ where
$\bar{a}>0.$ With a suitable choice of $L$, resp., $(L_p)_p$, we
have that this line lies in the domain $|Im (i\zeta)|<\frac{|Re
(i\zeta)|}{2}+\frac{1}{L},$ resp., $|Im (i\zeta)|<\frac{|Re
(i\zeta)|}{2}+\frac{1}{L_1},$ where we have the quoted estimates for
$P_L(-i\lambda),$ resp., $P_{L_p}(-i\lambda).$ Let us explain this
in the Beurling case with
more details. Choose any $L\in (0,\frac{1}{\bar{a}})$ and put \\
$K(t)=\frac{1}{2\pi i}\int_{\bar{a}-i\infty}^{\bar{a}+i\infty}\frac{e^{\lambda
t}}{P_{L}(-i\lambda)}d\lambda,\ t \geq 0.$ Then $K$ is an
exponentially bounded, continuous function defined on $[0,\infty)$
and we shall simply write \\ $K={\mathcal
L}^{-1}(\frac{1}{P_L(-i\lambda)}).$

Now, we will give the structural characterizations for
(UDSG)'s and exponential (UDSG)'s.
Some of these characterizations are proved in
\cite{k92},\cite{kosticknjiga}, \cite{mn},
\cite{ku113}, and \cite{me152}.
We will indicate this in Theorem \ref{struc}.  \vspace{0.2cm}

First, we list the statements: \vspace{0.2cm}

\begin{itemize}
\item[(a)] $A$ generates a (UDSG) of $*$-class $G$.
\item[(a)'] $A$ generates a  (EUDSG) of $*$-class $G$.
\item[(b)] $A$ generates a  (UDSG) of $*$-class $G$ such that, for
every $a>0,$ $G$ is of the form $G =P^{a}_L(-id/dt)S^{a}_K$ on
${\mathcal D}^{(M_p)}((-\infty, a))$ in $(M_p)$-case, (resp.,
$G=P^{a}_{L_p}(-id/dt)S^{a}_K$ on ${\mathcal D}^{\{M_p\}}((-\infty,
a))$ in $\{M_p\}$-case), where $S^{a}_K:(-\infty,a)\rightarrow
L(E,[D(A)])$ is continuous, $S^{a}_K(t)=0,\; t\leq 0.$
\item[(b)'] $A$ generates a (EUDSG) of $*$-class $G$  so that
$G$ is of the form $G=P_L(-id/dt)S_K$ on
${\mathcal{SE}}^{(M_p)}_{a}({\mathbb R})$ in $(M_p)$-case, (resp.,
$G=P_{L_p}(-id/dt)S_K$ in $\{M_p\}$-case), where $S_K: {\mathbb
R}\rightarrow L(E,[D(A)])$ is continuous, $S_K(t)=0,\; t\leq 0$ and
$e^{-at}||S_k(t)||\leq Ae^{M(k|t|)},$ for some $k>0$ and $A>0$, resp., for
every $k>0$ and  corresponding $A>0,$ $t\in \mathbb{R}.$
\item[(c)] For every $a>0$, $A$ is the generator  of a local
non-degenerate $K_{a}$-convoluted semigroup
$(S^{a}_{K_{a}}(t))_{t\in [0,a)},$ where $K_{a}={\mathcal
L}^{-1}(\frac{1}{P^{a}_L(-i\lambda)})$ in $(M_p)$-case, resp., $
K_{a}={\mathcal L}^{-1}(\frac{1}{P^{a}_{L_p}(-i\lambda)})$ in
$\{M_p\}$-case and $P^{a}_{L},$ resp., $P^{a}_{L_p},$ is an
ultradifferential operator of $*$-class such that for $0<a<b$ the
restriction of $P^{b}_{L}S^{b}_{K},$ resp., $P^{b}_{L_p}S^{b}_{K},$
on ${\mathcal D}^*((-\infty,a))$ is equal to $P^{a}_{L}S^{a}_{K},$
resp., $P^{a}_{L_p}S^{a}_{K}.$
\item[(c)']  $A$ is the generator of a global, exponentially bounded
non-degenerate $K$-convoluted semigroup $(S_{K}(t))_{t\geq 0},$
where $K={\mathcal L}^{-1}(\frac{1}{P_L(-i\lambda)})$ in
$(M_p)$-case, resp., $ K={\mathcal
L}^{-1}(\frac{1}{P_{L_p}(-i\lambda)})$ in $\{M_p\}$-case.
\item[(d)] There exists an ultradistribution fundamental solution of $\ast$-class for
$A$, denoted by $G,$ with the property ${\mathcal N}(G)=\{0\}.$
\item[(d)'] There exists an exponential ultradistribution fundamental solution of $\ast$-class
$G$ for $A$, with the property ${\mathcal N}(G)=\{0\}.$
\item[(e)] $\rho(A)\supset \Omega^*$ and
$$
||R(\lambda : A)||\leq Ce^{M(k|\lambda|)},\mbox{   } \lambda\in \Omega^{(M_p)},
$$
for some $k>0$ and $ C>0$ in $(M_p)$-case, resp.,
$$
||R(\lambda : A)||\leq C_ke^{M(k|\lambda|)},\mbox{   } \lambda \in
\Omega^{\{M_p\}},
$$
for every $k>0$ and a corresponding $C_k>0$ in $\{M_p\}$-case.
\item[(e)']   $\rho(A)\supset \{\lambda \in \mathbb{C} : Re \lambda>a\}$ and
$$
||R(\lambda : A)||\leq Ce^{M(k|\lambda|)},\mbox{   } Re \lambda>a,
$$
for some $a, k>0$ and $ C>0$ in $(M_p)$-case, resp.,
$$
||R(\lambda : A)||\leq C_ke^{M(k|\lambda|)},\mbox{   } Re \lambda>a,
$$
for every $k>0$ and a corresponding $a, C_k>0$ in $\{M_p\}$-case.
\end{itemize}


\begin{thm} \label{struc}
(a) $\Leftrightarrow$ (d); (a)' $\Leftrightarrow$ (d)'; (c)
$\Rightarrow$ (d); (c)' $\Rightarrow$ (d)'; (d) $\Rightarrow$ (e);
(d)' $\Rightarrow$ (e)'; if $(M_{p})$ additionally satisfies
$(M.3)$, then $(a)'$ $\Rightarrow$ (c)'.
\end{thm}

\begin{proof}

(a) $\Leftrightarrow$ (d):
This equivalence is proved in \cite{mn},
when ${\mathcal N} (G)\neq\{0\}$.
The statement (a) $\Rightarrow$ (d)
is direct consequence of ~\cite[Theorem 2 (c)]{kps}.
We give here the sketch of the proof of the opposite direction. Let
$G\in{\mathcal D}_{+}^{*'}({\mathbb R},L(E,[D(A)]))$
be an ultradistributional fundamental solution
of $*$-class for $A$. By the direct calculation we have
that $A$ is closable operator.

Let $\tilde{A}$ generates $G$.
If $(x,y)$ belongs to the closure of $A$, then there
exists a sequence $(x_n,y_n)_n$ in $A$ such that
$(x_n,y_n)\rightarrow(x,y)$, when $n\rightarrow\infty$,
in $E\times E$. Let $\phi\in{\mathcal D}_
0^{*}({\mathbb R})$ be fixed. For $\varphi\in
{\mathcal D}_0^{*}({\mathbb R})$ we have
$$\|G(\varphi)(G(-\phi')x-G(\phi)y)\|=$$ $$=
\|G(\varphi)[G(\phi')(x_n-x)-G(\phi')x_n+G(\phi)
(y_n-y)-G(\phi)y_n]\|=$$
$$=\|G(\varphi)[G(\phi')(x_n-x)+G(\phi)(y_n-y)]\|
\leq (\|G(\varphi*_0\phi')\|+\|G(\varphi*_0\phi\|)/ k\, ,$$
for $k\in{\mathbb N}$. So it follows $G(-\phi')x=G(\phi)y$
for all $\phi\in{\mathcal D}_0^{*}({\mathbb R})$.
Since $G$ is a ultradistribution fundamental
solution of $*$-class for $A$ we have $A\subset\tilde{A}$.
It implies that ${\mathcal D}_+^{'*}({\mathbb R},
[\overline{D({A})}])$ is an isomorphic to a
subspace of ${\mathcal D}_+^{'*}({\mathbb R},
[D(\tilde{A})])$. From the first part of the
theorem we have that $G$ is a fundamental
ultradistribution solution for $P:=\delta'
\otimes Id_{D[\tilde{A}]}-\delta\otimes{\tilde A}$.
So $G^{*}$ is an isomorphism from ${\mathcal D}_
+^{'*}({\mathbb R}, E)$ onto ${\mathcal D}_+^{'*}
({\mathbb R}, [D({A})])$ and onto
${\mathcal D}_+^{'*}({\mathbb R}, [D(\tilde{A})])$
which implies that ${\mathcal D}_+^{'*}({\mathbb R},
[D({A})])$=${\mathcal D}_+^{'*}({\mathbb R},
[D(\tilde{A})])$, so $[D(A)]=[D({\tilde A})]$.

The statement (a)' $\Leftrightarrow$ (d)' can be proved similarly
using that $G$ can be extended continuously
on ${\mathcal{ES}}^{*}({\mathbb R})$, \cite{kps}. \\

The proof of (d) $\Rightarrow$ (e) is given in \cite{me152}.

(d)' $\Rightarrow$ (e)'\cite{kosticknjiga} : We will
give a proof for Beurling case. The Roumeiu case is
quite similar. Let $G$ be a exponential fundamental ultradistribution
solution of $(M_p)$-class for $A$, i.e $G$ is a
fundamental ultradistribution solution and
$G\in {\mathcal{SE}}'^{(M_p)}_{\omega}(\mathbb{R},L(E))$ for $\omega\geq 0$.
Let $s>0$. We define a function $g\in{\mathcal E}^{(M_p)}
({\mathbb R})$ such that $g(t)=0$ for $t<-s$ and $g(t)=1$
for $t\geq0$. The definition of $\tilde{G}(\lambda):=
G(g(t)e^{-\lambda t}):=G(e^{-\omega t}(g(t)
e^{(\omega-\lambda)t}))$ have meaning since the
function $t\mapsto g(t)e^{(\omega-\lambda)t}$,
when $t\in{\mathbb R}$ and for all $\lambda\in{\mathbb C}$
such that $\mbox{Re}\lambda>\omega$, is in
${\mathcal S}^{(M_p)}({\mathbb R})$. Because
$G$ is a fundamental ultradistribution
solution for $A-\omega I$, for $\varphi
\in{\mathcal D}^{(M_p)}({\mathbb R})$, $x\in E$
we have that, $$(A-\omega I)G(e^{-\omega t}\varphi)x=
G(-e^{-\omega t}\varphi')x-\varphi(0)x\, .$$ Using that
${\mathcal D}^{(M_p)}({\mathbb R})$ is dense in
${\mathcal S}^{(M_p)}({\mathbb R})$, we get that the
previous equation holds for all ${\mathcal S}^{(M_p)}
({\mathbb R})$. Let we put $\varphi(t)=g(t)
e^{(\omega-\lambda)t}\in{\mathcal S}^{(M_p)}({\mathbb R})$.
Then supp$ G\subseteq[0,\infty)$ and we obtain:
$$A{\tilde G}(\lambda)x=AG(e^{-\lambda t}\varphi)x=
\lambda\tilde{G}(\lambda)x-\varphi(0)x, \quad \mbox{Re}
\lambda>\omega\, .$$ From this equation,
$(\lambda I -A)\tilde{G}(\lambda)x=x$, $x\in E$,
$\mbox{Re}\lambda>\omega$. ${\tilde G}(\lambda)
A\subseteq A{\tilde G}(\lambda)$ holds for
$\mbox{Re}\lambda>\omega$ we have ${\tilde G}(\lambda)
(\lambda I -A)x=x$, for $x\in D(A)$ and
$\mbox{Re}\lambda>\omega$. We put $\omega =a$
so we have proved the first part of the statement.
From the discussion above, it is clear that
$R(\lambda:A)x={\tilde G}(\lambda)x$, for $x\in E$,
$\mbox{Re}\lambda>a$. Using $(M.1)$ we obtain that
$$\|R(\lambda:A)\|=\|{\tilde G}(\lambda)\|=
\|G(e^{-\omega t}(g(t)e^{(\omega-\lambda)t}))\|\leq $$
$$\leq C'' \sup\limits_{t\in K}\frac{(g(t)
e^{(\omega-\lambda)t})^{(p)}}{M_p h^p}\leq C''
\sup\limits_{t\in K}\sum\limits_{j=0}^{p}C_j^p
\frac{g^{(p-j)}(t)\cdot
(e^{(\omega-\lambda)t})^{(j)}}{M_ph^p}\leq$$
$$\leq C'' \sup\limits_{t\in K}\sum
\limits_{j=0}^{p}C_j^p\frac{g^{(p-j)}(t)}
{M_{p-j}h^{p-j}}\cdot\frac{(\omega-\lambda)^j
e^{(\omega-\lambda)t}}{M_jh^j}\leq$$ $$\leq C'
\sup\limits_{t\in K}\sum\limits_{j=0}^{p}C_j^p
\frac{(\omega-\lambda)^je^{(\omega-\lambda)t}}{M_jh^j}\leq Ce^{M(k|\lambda|)}\, .$$

 (a)' $\Rightarrow$ (c)':

We will prove this assertion  in the
Beurling case by the use of already mentioned structural theorem for
elements of ${\mathcal{SE}}'^{(M_p)}_{a}(L(E)):$
$$G(\phi)=
\langle\phi, P_L(-id/dt)S(t))\rangle,\; \phi\in{\mathcal
S}^{(M_p)}({\mathbb R}),
$$
where, for an appropriate $k>0,$
$$
e^{-at}||S(t)||\leq e^{M(k|\xi|)},\ t\in {\mathbb R}.
$$
Fix an $x\in E.$ By Theorem ~\cite[Theorem 2 (c)]{kps},
$$
AG(\phi)x= -\langle \phi',P_L(-id/dt)S(t)x \rangle-\phi(0)x,\mbox{
for all }\phi \in {\mathcal S}^{(M_p)}({\mathbb R}).
$$
Since $1=P_L(-id/dt){\mathcal L}^{-1}(1/P_{L}(-i\cdot))$ in the
sense of ultradistributions, we have, for every $\phi \in {\mathcal
S}^{(M_p)}(\mathbb R),$
$$0=\langle \phi'(t),(P_L(-id/dt)A\int\limits_0^t S(s)x\, ds
 -P_L(-id/dt)S(t)x
$$
$$
+\int \limits^{t}_{0}{\mathcal L}^{-1}(1/P_{L}(-i\cdot))(s)x\, ds ) \rangle
$$
$$=\langle P_L(id/dt)\phi'(t),(A\int\limits_0^tS(s)x\, ds
-S(t)x +\int\limits^{t}_{0}{\mathcal L}^{-1}(1/P_{L}(-i\cdot))(s)x\, ds)
\rangle.
$$
Assume that $\psi\in{\mathcal D}(\mathbb R)$ and $\phi\in{\mathcal
S}^{(M_p)}({\mathbb R})$ so that $\psi=P_L(id/dt)\phi$ (cf. Lemma
\ref{izo}). This implies
\begin{equation}\label{pre1}
A\int\limits_0^tS(s)x\, ds -S(t)x +\int\limits^{t}_{0}{\mathcal
L}^{-1}(1/P_{L}(-i\cdot))(s)x\, ds=const,
\end{equation}
in the sense of Beurling ultradistributions on
$(0,\infty)$. We obtain that $const=0$ by putting $x=0$ in
(\ref{pre1}). Since the left side of (\ref{pre1}) is continuous on
${\mathbb R},$ we have
$$
A\int_0^tS(s)x\, ds =S(t)x-\Theta(t)x=0,\; \mbox{ where }\;
\Theta(t)=\int^{t}_{0}{\mathcal
L}^{-1}(1/P_{L}(-i\cdot))(s)\, ds,$$

for all $t\geq 0.$ This completes the proof of (a)' $\Rightarrow$
(c)'.

Let us show (c) $\Rightarrow$ (d) in the Beurling case.
The proof of (c)' $\Rightarrow$ (d)' is  similar. Define $G $
on ${\mathcal
 D}^{(M_{p})}((-\infty,a)),$ for all $a>0,$ by
$$
G:=P^{a}_{L}(-id/dt)S^{a}_{K_{a}},\; \mbox{ where }
P^a_L=\sum_{p=0}^{\infty} a_p(d/dt)^p.
$$
Then $G$ is a continuous
linear mapping from ${\mathcal
 D}^{(M_{p})}({\mathbb R})$ into $L(E)$ which commutes with
$A$. Moreover, supp$G \subset [0,\infty).$ Let $\phi\in {\mathcal
D}^{(M_p)}((-\infty,a))$ and $x\in E.$ We have,
$$
G(-\phi')x-AG(\phi)x=-\sum_{p \geq 0}a_{p}(-i)^{p}\int
\limits^{a}_{0} \phi^{(p+1)}(s)S^{a}_{K_{a}}(s)x\, ds
$$
$$- \sum_{p \geq 0}a_{p}(-i)^{p}\int \limits^{a}_{0}
\phi^{(p)}(s)AS^{a}_{K_{a}}(s)x\, ds=-\sum_{p \geq 0}a_{p}(-i)^{p}\int
\limits^{a}_{0} \phi^{(p+1)}(s)S^{a}_{K_{a}}(s)x\, ds
$$
$$
+\sum_{p \geq 0}a_{p}(-i)^{p}\int \limits^{a}_{0}
\phi^{(p+1)}(s)(S^{a}_{K_{a}}(s)x-\Theta_{a}(s)x)\, ds=
$$
$$
=\sum_{p \geq 0}a_{p}(-i)^{p}\int \limits^{a}_{0}
\phi^{(p)}(s)K_{a}(s)x\, ds=\phi(0)x.
$$
Hence, $G \in {\mathcal D}'^{(M_p)}(\mathbb{R},L(E,[D(A)]))$ is an
ultradistribution fundamental solution for $A$. Clearly, ${\mathcal
N}(G)=\{0\}.$
\end{proof}




\begin{thebibliography}{1}
\bibitem{a11} {W. Arendt},  \textit{Vector-valued Laplace transforms and Cauchy problems.}
Israel J. Math. \textbf{59} (1987), 327--352.

\bibitem{a22} {W. Arendt, O. El-Mennaoui\and V. Keyantuo},  \textit{Local integrated
semigroups: evolution with jumps of regularity.} J. Math. Anal.
Appl. \textbf{186} (1994), 572--595.

\bibitem{a43} {W. Arendt, C. J. K. Batty, M. Hieber \and F. Neubrander},
\textit{Vector-valued Laplace Transforms and Cauchy Problems.}
Birkh\" auser Verlag, 2001.



\bibitem{b41} {R. Beals},  \textit{On the abstract Cauchy problem.} J. Funct. Anal.
\textbf{10} (1972), 281--299.

\bibitem{b42} {R. Beals},  \textit{Semigroups and abstract Gevrey spaces.} J. Funct. Anal.
\textbf{10} (1972), 300-308.



\bibitem{cha} {J. Chazarain},  \textit{Probl\' emes de Cauchy abstraites et applications \' a
quelques probl\' emes mixtes.} J. Funct. Anal. \textbf{7} (1971),
386--446.

\bibitem{ci1} {I. Cior\u anescu},  \textit{Beurling spaces of class $(M\sb{p})$ and
ultradistribution semi-groups.} Bull. Sci. Math. \textbf{102}
(1978), 167--192.









\bibitem{l115} {M. Hieber},
\textit{Integrated semigroups and differential
operators on $L^{p} $ spaces.} Math. Ann, \textbf{29} (1991), 1- 16.


\bibitem{keya} {V. Keyantuo}, \textit{Integrated semigroups and related partial differential
equations.} J. Math. Anal. Appl. \textbf{212} (1997), 135--153.




\bibitem{k91} {H. Komatsu},  \textit{Ultradistributions, I. Structure theorems and a
characterization.} J. Fac. Sci. Univ. Tokyo Sect. IA Math.
\textbf{20} (1973), 25--105.

\bibitem{k82} {H. Komatsu},  \textit{Ultradistributions, III. Vector valued
ultradistributions the theory of kernels.} J. Fac. Sci. Univ. Tokyo
Sect. IA Math. \textbf{29} (1982), 653--718.

\bibitem{k92} {H. Komatsu},  \textit{Operational calculus and semi-groups of operators.}
Functional Analysis and Related topics (Kioto), Springer, Berlin,
213-234, 1991.

\bibitem{ko98} {M. Kosti\' c},  \textit{C-Distribution semigroups.} Studia
Math. \textbf{185} (2008), 201--217.

\bibitem{kosticknjiga}  {M. Kosti\' c},
\textit{Generalized semigroups and cosine functions.} Mathematical Institute,
Belgrade, 2011.


\bibitem{mn} {M. Kosti\' c \and S. Pilipovi\' c},  \textit{Global convoluted semigroups.}
Math. Nachr., \textbf{280}, No. 15 (2007), 1727--1743.

\bibitem{kps} {M. Kosti\' c \and S. Pilipovi\' c},
\textit{Ultradistribution semigroups.} Siberian Math. J.,
\textbf{53}, No. 2 (2012), 232-242.

\bibitem{ku101} {P. C. Kunstmann},
\textit{Stationary dense operators and generation of
non-dense distribution semigroups.} J. Operator Theory \textbf{37}
(1997), 111--120.


\bibitem{ku112} {P. C. Kunstmann},  \textit{Distribution semigroups and abstract Cauchy
problems.} Trans. Amer. Math. Soc. \textbf{351} (1999), 837--856.

\bibitem{ku113} {P. C. Kunstmann},  \textit{Banach space valued ultradistributions and
applications to abstract Cauchy problems.} preprint.



\bibitem{l114} {M. Li, F. Huang \and Q. Zheng},  \textit{Local integrated $C$-semigroups.}
Studia Math. \textbf{145} (2001), 265--280.



\bibitem{me151}{I. V. Melnikova} \textit{The method of integrated semigroups for Cauchy problems in Banach spaces.}
Siberian Math. J., \textbf{40}, No. 1 (1999), 119--129.


\bibitem{me149}{I. V. Melnikova} \textit{Regularized solutions to Cauchy problems well posed in the extended sense.}
Integral Transforms Spec. Funct. , \textbf{17} No. 2--3 (2006), 185--191.

\bibitem{me152} {I. V. Melnikova \and A. I. Filinkov},  \textit{Abstract Cauchy Problems:
Three Approaches.} Chapman \& Hall/CRC, 2001.

\bibitem{n181} {F. Neubrander},  \textit{Integrated semigroups and their applications to the
abstract Cauchy problem.} Pacific J. Math. \textbf{135} (1988),
111--155.




\bibitem{pilip} {S. Pilipovi\' c},  \textit{Tempered ultradistributions.} Boll.
Un. Mat. Ital. \textbf{7} (1988), 235-251.

\bibitem{sp94} {S. Pilipovi\' c},  \textit{Characterizations of bounded sets in spaces of
ultradistributions.} Proc. Amer. Math. Soc. \textbf{20} (1994),
1191-1206.



\bibitem{w241} {S. Wang},  \textit{Quasi-distribution semigroups and integrated semigroups.}
J. Funct. Anal. \textbf{146} (1997), 352--381.

(

\end{thebibliography}
\end{document}